\documentclass[12pt]{article}
\usepackage{amsmath}
\usepackage{graphicx,psfrag,epsf}
\usepackage{enumerate}
\usepackage{natbib}
\usepackage{amsmath,amsthm,amsfonts,epsfig,amssymb,natbib,eucal,eufrak}
\usepackage{booktabs}
\usepackage{multirow}
\usepackage{siunitx}
\usepackage{qtree}
\usepackage{color}
\usepackage{hyperref}

\usepackage{float}
\floatstyle{plaintop}
\restylefloat{table}

\newtheorem{claim}{Claim}

\newtheorem{cor}{Corollary}
\newtheorem{pro}{Proposition}

\newcommand{\blind}{0}

\begin{document}

\def\spacingset#1{\renewcommand{\baselinestretch}%
{#1}\small\normalsize} \spacingset{1}


\if0\blind
{
  \title{\bf Hitting a prime in 2.43 dice rolls\\ (on average)}
  \author{Noga Alon\\
    Department of Mathematics, Princeton University\\ Princeton, NJ 08544, USA\\
    and\\
    Schools of Mathematics and
    Computer Science\\ Tel Aviv University\\ Tel Aviv 6997801, Israel\\
    \\
    \\
    Yaakov Malinovsky\thanks{Corresponding author}\\
    Department of Mathematics and Statistics\\ University of Maryland, Baltimore County\\ Baltimore, MD 21250, USA}

  \maketitle
} \fi

\if1\blind
{
  \bigskip
  \bigskip
  \bigskip
  \begin{center}
    {\LARGE\bf Title}
\end{center}
  \medskip
} \fi

\begin{abstract}
What is the number of rolls of fair $6$-sided dice until the first time
the total sum of all rolls is a prime?
We compute the expectation
and the variance of this random variable up to an additive error of
less than $10^{-4}$. This is a solution to a
puzzle suggested  by \cite{D2017} in the Bulletin of the Institute of Mathematical Statistics,
where the published solution is incomplete.
The proof is simple, combining
a basic dynamic programming algorithm with a quick Matlab
computation and basic facts about the distribution of primes.
\end{abstract}

\noindent%
{\it Keywords: dynamic-programming, prime number theorem, stopping time}
\vfill

\newpage
\spacingset{1.45} 

\section{The Problem and Monte-Carlo Simulation}
The following puzzle appears in the Bulletin of the Institute of
Mathematical Statistics \citep{D2017}:
Let $X_1,X_2,\ldots$ be independent uniform random variables
on the integers $1,2,\ldots,6$,
and define $S_n=X_1+\ldots+X_n$ for $n=1,2,\ldots$.
Denote by $\tau$ the discrete time in which $S_n$ first hits
the set of prime numbers $P$:
\begin{equation*}
\label{eq:aaa}
\tau=\min\left\{n\geq 1: S_n \in P\right\}.
\end{equation*}
The contributing editor \citep{D2017} provides a lower bound of
$2.34$ for the expectation $E(\tau)$ and mentions the following heuristic
approximation for it: $E\left(\tau\right)\approx 7.6.$
He also adds that it is unknown whether $\tau$ has a finite variance.

In this note, we compute the value of $E(\tau)$ up to an additive error
of less than $10^{-7}$, showing it is much closer to the
lower bound mentioned above than to $7.6$. We also show the variance
is finite and compute its value up to an additive error of less than
$10^{-4}$. It will be clear from the discussion that it is not
difficult to get a better approximation for both
quantities by increasing the amount of computation performed.

Before describing the rigorous argument, we present in Table 1 below
the outcomes of Monte-Carlo simulations of the process.

\begin{table}[H]
\label{eq:t}
\caption{Monte-Carlo simulations}
\begin{center}
\begin{tabular}{lllllll}
  number of repetitions  & $mean(\tau)$ & $variance(\tau)$ & $max(\tau)$\\
  \hline
  $10^6$     &  2.4316    & 6.2735      & 49\\
  $2\times 10^6$   &  2.4274    & 6.2572      &67\\
  $3\times 10^6$   &  2.4305    & 6.2372      &70\\
  $5\times 10^6$   &  2.4287    & 6.2418      &64\\
  $10^7$     &  2.4286     &6.2463       &65\\
\end{tabular}
\end{center}
\end{table}
We provide Matlab code for the Monte-Carlo simulation in Web Appendix A.

In the next sections, we proceed with a rigorous computation of $E(\tau)$ and $Var(\tau)$
up to an additive error smaller than $1/10,000$. Not surprisingly,
this computation shows that the simulations supply accurate values.

\section{Expectation and Variance of the Hitting Time}

First, we present the formulas for calculating the expectation and variance of the hitting time $\tau$ as a function of the probability that $\tau$ equals or exceeds a certain value $k$, for $k=1,2,3, \ldots$, which we denote by $p(k)=P(\tau\geq k).$
We have
\begin{equation}
\label{eq:E}
E\left(\tau\right)=\sum_{k\geq 1}p(k)
\end{equation}
and
\begin{equation}
\label{eq:EE}
E\left(\tau^2\right)=\sum_{k\geq 1}(2k-1)p(k).
\end{equation}

We remind in Appendix A how to obtain the formulas \eqref{eq:E} and \eqref{eq:EE}.

Obviously, by the definition of variance we have
\begin{equation}
\label{eq:v}
Var(\tau)=E(\tau^2)-\left[E(\tau)\right]^2.
\end{equation}
In section \ref{sec:DP} we develop a
dynamic programming algorithm to compute $p(k)$ exactly and use the first 1000 values ($k=1,\ldots,1000$)
to estimate $E(\tau)$ and $Var(\tau)$ with reference to expressions \eqref{eq:E}, \eqref{eq:EE}, and \eqref{eq:v}.

\section{Dynamic Programming Algorithm and Estimates}
\label{sec:DP}
In this section we develop a dynamic programming algorithm to compute
the first $K$ values $p(1), p(2),\ldots, p(K)$, then use these values to estimate $E(\tau)$ and $Var(\tau)$.
\subsection{Dynamic Programming Algorithm}

For each integer $k\geq 1$ and for each non-prime $n$ satisfying
$k\leq n\leq 6k$, let $p(k,n)$ denote the probability that
$X_1+\cdots+X_k=n$ and that for every $i<k$, $X_1+\cdots+X_i$ is non-prime.
Fix a parameter $K$ (in our computation, we later take
$K=1000$).
By the definition of $p(k,n)$ and the rule of total probability,
we have the following dynamic programming (DP) algorithm for computing
$p(k,n)$ precisely for all $1 \leq k \leq K$ and $k \leq n \leq 6k$:
\begin{enumerate}
\item[1.]  $p(1,1)=p(1,4)=p(1,6)=1/6.$
\item[2.]
For $k=2,\ldots,K$ and for any non-prime $n$ between $k$ and $6k$,
\begin{equation}
\label{eq:TP}
p(k,n)=\frac{1}{6}\sum_{i}p(k-1,n-i),
\end{equation}
\end{enumerate}
where the sum ranges over all $i$ between $1$ and $6$ so that
$n-i$ is non-prime.\\
From the definitions of $p(k)$ and $p(k, n)$, we obtain the following identity:
\begin{align}
\label{eq:id}
&
p(k+1)=\sum_{\left\{n:\,\, k \leq n \leq 6k\right\}}
p(k,n).
\end{align}

We apply this DP algorithm to get the values $p(k, n)$ for $k=1,\ldots, K$. Then, using identity \eqref{eq:id}, we obtain the values $p(1), p(2),\ldots, p(K)$, which consequently provide us the partial sums in \eqref{eq:E} and \eqref{eq:EE}. These partial sums are the lower bounds to $E(\tau)$ and $E(\tau^2)$, which allows us to estimate the expectation and variance of the hitting time, which we discuss in the next section.

\subsection{Estimators of the Expectation and Variance of Hitting Time}

Denote by $E_{K}$ and $E^{(2)}_{K}$ the estimators (lower bounds)
of $E(\tau)$ and $E(\tau^2)$
based on the respective values of $p(k)$ for the first $K$ values in \eqref{eq:E} and \eqref{eq:EE}:
\begin{align}
\label{eq:twoM}
&
E_{K}=\sum_{k=1}^{K}p(k)\,\,\,\,\,\text{and}\,\,\,\,\, E^{(2)}_K=\sum_{k=1}^{K}(2k-1)p(k).
\end{align}

We estimate $E(\tau)$ by $E_{K}$ and consequently $Var(\tau)$ by $V_{K}$, defined as follows
\begin{align*}
&
V_K=E^{(2)}_{K}-(E_K)^2.
\end{align*}

Setting $K=1000$ and applying the dynamic programming algorithm in Matlab (provided in Web Appendix B), with
an execution time of less than five seconds, we obtain
$E_{1000}=2.4284$ and $V_{1000}= 6.2427.$

The "quality" of these estimators can be measured as the difference between
$E(\tau)$ and $Var(\tau)$ and their corresponding estimators.
The define these distances (remainders) as
\begin{align}
\label{eq:nR}
&
RE_{K}=E(\tau)-E_K\,\,\,\,\text{and}\,\,\,\, RV_{K}=Var(\tau)-V_K.
\end{align}

In the next section, we bound $RE_K$ and $RV_K$ and show that
for $K=1000$, $RE_{1000}<10^{-7}$  and  $RV_{1000}<10^{-4}.$

\section{Bounding the Remainders}
\label{sec:UB}

In this section, we provide the bounds for the remainder terms defined in \eqref{eq:nR}.
To accomplish this, we will use basic facts about the distribution of primes and
prove the following simple result by induction on $k$.

\begin{pro}
For every $k$ and for every non-prime $n$,
\begin{equation}
\label{eq:cl1}
p(k,n)<\frac{1}{3}\left(\frac{5}{6}\right)^{\pi(n)},
\end{equation}
where $\pi(n)$ is the number of primes smaller than $n$.
\end{pro}

\begin{proof}
Note first that \eqref{eq:cl1} holds for $k=1$, as $1/6=p(1,6)<(1/3) (5/6)^3$,
$1/6=p(1,4)<(1/3)(5/6)^2$  and $1/6=p(1,1)<(1/3)(5/6)^{0}$, with
room to spare.
Assuming the inequality holds for $k-1$ (and every relevant $n$)
we prove it for $k$. Suppose there are $q$ primes in the set
$\left\{n-6,\ldots,n-1\right\}$, then $\pi(n-i)\geq \pi(n)-q$
for all non-prime $n-i$ in this set. Thus, by the induction hypothesis,
and using \eqref{eq:TP}, we obtain
\begin{align*}
&
p(k,n) \leq \frac{1}{6}(6-q)\frac{1}{3}
\left(\frac{5}{6}\right)^{\pi(n)-q}
\leq \left(\frac{5}{6}\right)^q \frac{1}{3}
\left(\frac{5}{6}\right)^{\pi(n)-q}
=\frac{1}{3} \left(\frac{5}{6}\right)^{\pi(n)}.
\end{align*}
\end{proof}

By the prime number theorem (cf.,  e.g., \cite{HW2008}),
for every $n>1000$
${\displaystyle \pi(n) > 0.9 \frac{n}{\ln n}}$ (again, with room to spare).
Therefore, from  the above estimate, we arrive at the following result.

\begin{cor}
\label{eq:cor1}
For every $k>1000$ and every non-prime $n (n\geq k)$,
\begin{equation*}
\label{eq:cl2}
p(k,n)<\frac{1}{3}\left(\frac{5}{6}\right)^{0.9 \frac{n}{\ln n}}.
\end{equation*}
\end{cor}

Corollary \ref {eq:cor1} is the crucial result in obtaining the upper bounds of the remainders $RE_{1000}$ and $RV_{1000}$, which are given in the following proposition.
\begin{pro}
\label{eq:bound}
The remainder terms, which are defined in \eqref{eq:nR}, are bounded as follows:
\begin{align*}
&
{(a)}\,\, RE_{1000}<10^{-7},\\
&
(b)\,\, RV_{1000}<10^{-4}.
\end{align*}
\end{pro}
\begin{proof}
Recall that
\begin{align*}
{\displaystyle P(\tau\geq k+1)=p(k+1)=\sum_{\left\{n:\,\, k \leq n \leq 6k\right\}}
p(k,n).}
\end{align*}
For part (a), we have
\begin{align}
\label{eq:in1}
&
RE_{1000}=\sum_{k>1000}P(\tau\geq k)=\sum_{k>999}P(\tau\geq k+1)
=\sum_{k>999} \sum_{\left\{n:\,\, k \leq n \leq 6k\right\}}
p(k,n) \nonumber\\
&
<
 \sum_{k >999} \sum_{\left\{n:\,\, k \leq n \leq 6k\right\}}
\frac{1}{3} \left(\frac{5}{6}\right)^{0.9 n/ \ln n}
=\sum_{n\geq 1000}\sum_{k=\max(1000,n/6)}^{n}\frac{1}{3}
\left(\frac{5}{6}\right)^{0.9 n/ \ln n}\nonumber \\
 &
< \sum_{n\geq 1000}\sum_{k=1000}^{n}\frac{1}{3}
\left(\frac{5}{6}\right)^{0.9 n/ \ln n}=
 \sum_{n\geq 1000}(n-999)\frac{1}{3}
\left(\frac{5}{6}\right)^{0.9 n/ \ln n},
\end{align}
where the first inequality is obtained from Corollary \ref{eq:cor1}.

Define
\begin{align}
\label{eq:f}
&
f(n)=(n-999)\frac{1}{3} \left(\frac{5}{6}\right)^{0.9 n/ \ln n},
\end{align}
where $n$ is an integer $\geq 1000$.
The first part of Proposition \ref{eq:bound} follows by noting ${\displaystyle \sum_{n\geq 1000}f(n)<10^{-7}}$, shown in Appendix \ref{eq:A}.
\bigskip

For part (b),
\begin{align}
&
R^{(2)}_{1000}:=\sum_{k>1000}(2k-1)P(\tau\geq k)
=\sum_{k>1000} \sum_{\left\{n:\,\, k-1 \leq n \leq 6(k-1)\right\}} (2k-1)p(k-1,n) \nonumber\\
&
<\sum_{k>1000} \sum_{\left\{n:\,\, k-1 \leq n \leq 6(k-1)\right\}}
(2k-1)\left(\frac{5}{6}\right)^{0.9 n/ \ln n}
=
\sum_{n\geq 1000}\frac{1}{3} \left(\frac{5}{6}\right)^{0.9 n/ \ln n}
\sum_{k=\max(1001,n/6+1)}^{n+1}(2k-1)\nonumber \\
&
<
\sum_{n\geq 1000}\frac{1}{3} \left(\frac{5}{6}\right)^{0.9 n/ \ln n}
\sum_{k=1001}^{n+1}(2k-1)
=
\sum_{n\geq 1000}\frac{1}{3} \left(\frac{5}{6}\right)^{0.9 n/ \ln n}
\left[(n+1)^2-1000^2\right],
\label{eq:99}
\end{align}
where the first inequality is also obtained from Corollary \ref{eq:cor1}.

Denote by
\begin{align}
\label{eq:g}
&
g(n)=\left[(n+1)^2-1000^2\right]\frac{1}{3}
\left(\frac{5}{6}\right)^{0.9 n/ \ln n},
\end{align}
where $n$ is an integer $\geq 1000$.
Likewise, the second part of Proposition \ref{eq:bound} follows by noting
${\displaystyle \sum_{n\geq 10,000}g(n)<3.4\times 10^{-68}}$, shown in Appendix \ref{eq:B}.

\noindent
Combining this with \eqref{eq:99}, we obtain

\begin{align}
\label{eq:nB}
&
R^{(2)}_{1000}<\sum_{n=1000}^{9999}g(n)+\sum_{n\geq 10,000}g(n)
< 8.5\times 10^{-5}+3.4\times 10^{-68}<1/10,000.
\end{align}

Now, from \eqref{eq:twoM} and \eqref{eq:nR} it follows that $RV_{K}=R^{(2)}_{K}-2E_K(RE_K)-(RE_K)^2$,
where ${\displaystyle R^{(2)}_{K}:=\sum_{k>K}(2k-1)p(k)}$.
Combining this with \eqref{eq:nB} and part (a) of Proposition \ref{eq:bound}, we conclude
that $RV_{1000}<10^{-4}$, i.e., the error of the variance estimation based on
the first $1000$ values of $k$
is  below $1/10,000$.
\end{proof}
\section{Final Remarks}
The problem considered in this work ties together prime numbers and the random model. Although the solution is tailored to this specific situation, it offers a method for studying problems of this type. Also, this paper can be used as a motivator in upper-level undergraduate or graduate classes by introducing and illustrating the power in combining a simple dynamic programming algorithm with a quick computer-aided
computation and basic facts about the distribution of primes.

\section*{Appendix}
\appendix
\section{}
Recall that the probability $p(k)$ was defined as $p(k)=P(\tau\geq k).$
We obtain \eqref{eq:E} from
\begin{align*}
&
E\left(\tau\right)=\sum_{k\geq 1}\tau P\left(\tau=k\right)=
\sum_{k\geq 1}k \left[P\left(\tau\geq k\right)-P\left(\tau\geq k+1\right)\right]\\
&
=\sum_{k\geq 1}k P\left(\tau\geq k\right)-\sum_{k\geq 1}(k+1) P\left(\tau\geq k+1\right)
+\sum_{k\geq 1}P(\tau\geq k+1)\\
&
=\sum_{k\geq 1}k P\left(\tau\geq k\right)-\sum_{k\geq 0}(k+1) P\left(\tau\geq k+1\right)
+\sum_{k\geq 0}P(\tau\geq k+1)\\
&
=\sum_{k\geq 0}P(\tau\geq k+1)=\sum_{k\geq 1}P(\tau\geq k)=\sum_{k\geq 1}p(k).
\end{align*}
In a similar manner, we obtain \eqref{eq:EE} from
\begin{align*}
&
E\left(\tau^2\right)=\sum_{k\geq 1}k^2 P\left(\tau=k\right)=\sum_{k\geq 1}k^2\left[P\left(\tau\geq k\right)-P\left(\tau\geq k+1\right)\right]\\
&
=\sum_{k\geq 1}k^2 P\left(\tau\geq k\right)-\sum_{k\geq 1}(k+1-1)^2 P\left(\tau\geq k+1\right)\\
&
=\sum_{k\geq 1}k^2 P\left(\tau\geq k\right)-\sum_{k\geq 1}(k+1)^2 P\left(\tau\geq k+1\right)
+\sum_{k\geq 1}(2k+1) P\left(\tau\geq k+1\right)
\\
&
=\sum_{k\geq 1}k^2 P\left(\tau\geq k\right)-\sum_{k\geq 0}(k+1)^2 P\left(\tau\geq k+1\right)
+\sum_{k\geq 0}(2k+1) P\left(\tau\geq k+1\right)
\\
&
=
\sum_{k\geq 0}\left[2(k+1)-1\right] P\left(\tau\geq k+1\right)=\sum_{k\geq 1}(2k-1) P\left(\tau\geq k\right)=\sum_{k\geq 1}(2k-1) p\left(k\right).
\end{align*}

\section{}Recall that the function $f(n)$ was defined in \eqref{eq:f} as
${\displaystyle f(n)=(n-999)\frac{1}{3} \left(\frac{5}{6}\right)^{0.9 n/ \ln n}}.$
\begin{claim}
${\displaystyle \sum_{n\geq 1000}f(n)<10^{-7}}.$
\end{claim}
\label{eq:A}
\begin{proof}
We investigate the behavior of the function $f(n)$:
\begin{itemize}
\item[(i)]
For
$n \geq 1000$ the function
$f(n)$ has a unique maximum at $n=1049$. \\
To see this, it suffices to compute $f(n)$ precisely for all
$1000 \leq n \leq 1049$ and verify that the maximum in this range is attained at $n=1049$.
Taking the derivative of $f(n)$ with respect to $n$ yields that the derivative is negative for $n\geq 1049$.

\item[(ii)]
For any $n\geq 1049$, ${\displaystyle f(n+13\ln n)/f(n)<1/2}$.\\
To see this, it suffices to verify that for any $n\geq 1049$,
\begin{align*}
&
\ln\left(f(n+13\ln n)/f(n)\right)
=\ln\left(1+\frac{13\ln n}{n-999}\right)+0.9\ln\left(\frac{5}{6}\right)
\left\{\frac{n+13\ln n}{\ln\left(n+13\ln n\right)}-\frac{n}{\ln(n)}\right\}\\
&
<\ln\left(1+\frac{13\ln 1049}{1049-999}\right)+0.9\ln\left(\frac{5}{6}\right)
\left\{\frac{n+13\ln n}{\ln\left(n+13\ln n\right)}-\frac{n}{\ln(n)}\right\}
:=r(n).
\end{align*}
Taking the derivative of $r(n)$ with respect to $n$ yields that the derivative is negative for any $n\geq 1049$. Also, it is straightforward to verify that $r(1049)<\ln(1/2)$.
Therefore, the function $r(n)$ is strictly decreasing for $n\geq 1049$ and so we obtain the inequality  $r(n)<r(1049)<\ln(1/2)$. The desired inequality follows as the function ${\displaystyle e^z}$ is monotone increasing.
\end{itemize}

From (i), it follows that
\begin{align}
\label{eq:f1}
&\sum_{n=1000}^{1049}f(n)<\sum_{n=1000}^{1049}f(1049)=50f(1049).
\end{align}

From (i) and (ii), it follows that for any $n\geq 1049$,
\begin{align}
\label{eq:In1}
&
f(n+13\ln n)<1/2f(n)<1/2f(1049).
\end{align}
Let ${\displaystyle d_1=\lfloor13\ln 1049\rfloor}$, and ${\displaystyle d_2=\lceil13\ln 1049\rceil}$.
Using the fact that the function $f(n)$ is strictly decreasing for $n\geq 1049$ and repeatedly applying \eqref{eq:In1}, we obtain
\begin{align}
\label{eq:f2}
&
\sum_{n\geq 1050}f(n)=\sum_{j\geq 0}\sum_{n=1050+jd_2}^{1050+jd_2+d_1}f(n)<
\sum_{j\geq 0}d_1 \left[f(1050+jd_2)\right]<
d_1 \sum_{j\geq 0}f(1049+j(13\ln 1049))\nonumber
\\
&
=d_1 \left[f(1049)+f(1049+13\ln 1049)+f(1049+2\times13\ln 1049)+f(1049+3\times13\ln 1049)+\cdots\right]\nonumber
\\
&
<d_1\left[f(1049)+1/2f(1049)+1/4f(1049)+1/8f(1049)+\cdots\right]
=d_1f(1049)\sum_{j\geq 0}\frac{1}{2^j}\nonumber\\
&
<(13\ln 1049)f(1049)\times 2.
\end{align}

Combining \eqref{eq:in1}, \eqref{eq:f1}, and \eqref{eq:f2}, we obtain

\begin{align*}
\label{eq:RE}
&
\sum_{n\geq 1000}f(n)
=\sum_{n=1000}^{1049}f(n) +\sum_{n\geq 1050}f(n)\nonumber
<
50f(1049)+2(13\ln 1049) f(1049)<7\times10^{-8}.
\end{align*}
\end{proof}

\section{}Recall that the function $g(n)$ was defined in \eqref{eq:g} as
$g(n)=\left[(n+1)^2-1000^2\right]\frac{1}{3}
\left(\frac{5}{6}\right)^{0.9 n/ \ln n}$.
\begin{claim}
$\sum_{n\geq 10,000}g(n)<3.4\times 10^{-68}.$
\end{claim}
\label{eq:B}
\begin{proof}
We investigate the behavior of the function $g(n)$:
\begin{itemize}
\item[(iii)]
For
$n \geq 10,000$ the function
$g(n)$ has a unique maximum at $n=10,000$. \\
To see this, we verify that the function $g(n)$ is strictly decreasing for $n\geq 10,000$ by taking the derivative of $g(n)$ with respect to $n$, which yields that the derivative is negative.

\item[(iv)]
For any $n\geq 10,000$ ${\displaystyle g(n+13\ln n)/g(n)<1/2}$.\\
To see this, it suffices to verify that for any $n\geq 10,000$,
\begin{align*}
&
\ln\left(g(n+13\ln n)/g(n)\right)
=\ln\left(1+\frac{26 (n+1)\ln n+13^2\ln^2(n)}{n^2-1000^2}\right)\\
&
+0.9\ln\left(\frac{5}{6}\right)
\left\{\frac{n+13\ln n}{\ln\left(n+13\ln n\right)}-\frac{n}{\ln(n)}\right\}\\
&
<\ln\left(1+\frac{26(10,001)\ln 10,000+13^2\ln^2(10,000)}{10,000^2-1000^2}\right)+0.9\ln\left(\frac{5}{6}\right)
\left\{\frac{n+13\ln n}{\ln\left(n+13\ln n\right)}-\frac{n}{\ln(n)}\right\}\\
&
:=l(n).
\end{align*}
Taking the derivative of $l(n)$ with respect to $n$ yields that the derivative is negative for any $n\geq 10,000$. Also, it is straightforward to verify that $l(10,000)<\ln(1/2)$.
Therefore, we obtain the inequality  $l(n)<l(10,000)<\ln(1/2)$. The desired inequality follows as the function ${\displaystyle e^z}$ is monotone increasing.
\end{itemize}

From (iii) and (iv), it follows that for any $n\geq 10,000$,
\begin{equation}\label{eq:I22}
g(n+13\ln n)<1/2g(n)<1/2g(10,000).
\end{equation}
Let ${\displaystyle l_1=\lfloor 13\ln 10,000\rfloor}$, and ${\displaystyle l_2=\lceil 13\ln 10,000\rceil}$.
Using the fact that the function $g(n)$ is strictly decreasing for $n\geq 10,000$, and repeatedly applying \eqref{eq:I22}, we obtain

\begin{align*}
\label{eq:g2}
&
\sum_{n\geq 10,000}g(n)=\sum_{j\geq 0}\sum_{n=10,000+jl_2}^{10,000+jl_2+l_1}g(n)<
\sum_{j\geq 0}l_1 g(10,000+jl_2)
<
\sum_{j\geq 0}l_1 g(10,000+j13\ln 10,000)\\
&
=l_1\left[g(10,000)+g(10,000+13\ln 10,000)+g(10,000+2\times13\ln 10,000)+\cdots\right]\\
&
<l_1g(10,000)\sum_{j\geq 0}\frac{1}{2^j}
<(13\ln 10,000)g(10,000)\times 2<3.4\times 10^{-68}.
\end{align*}

\end{proof}

\section*{Supplementary Materials}
\begin{description}
\item[Web Appendix A] Matlab code for the Monte-Carlo simulation.
\item[Web Appendix B] Matlab code for the dynamic-programming algorithm.
\end{description}

\section*{Acknowledgments}
We wish to thank the Editor, the Associate Editor and the two referees for helpful comments and suggestions.

\section*{Funding}
Research of Noga Alon is supported in part by
NSF grant DMS-2154082 and by BSF grant 2018267.
Research of Yaakov Malinovsky is supported in part by BSF grant 2020063.

\end{document}